\documentclass[11pt,letterpaper,reqno]{amsart}
\usepackage{tikz}
\usetikzlibrary{positioning, shapes.geometric, arrows.meta}
\usepackage{amssymb}
\usepackage{amsmath}
\usepackage{amsthm}
\usepackage{amsfonts}
\usepackage{bbm}
\usepackage{enumitem} 
\usepackage{pgfplots}
\pgfplotsset{compat=1.18} 
\usepackage{booktabs}

\usepackage{graphicx}
\usepackage[T1]{fontenc}
\usepackage{doi}
\usepackage{float} 
\usepackage{tikz}
\usepackage{pgfplots}
\pgfplotsset{compat=1.18}
\usetikzlibrary{arrows.meta}

\usepackage{bookmark}
\usepackage{hyperref}
\hypersetup{pdfstartview={FitH}}

\newtheorem{thm}{Theorem}[section]
\newtheorem{lem}[thm]{Lemma}
\newtheorem{prop}[thm]{Proposition}

\newtheorem{conj}[thm]{Conjecture}
\theoremstyle{definition}

\newtheorem{rem}[thm]{Remark}
\newtheorem{defin}[thm]{Definition}
\newtheorem{cons}[thm]{Construction}
\numberwithin{equation}{section}

\makeatother

\DeclareMathOperator{\Eq}{Eq}
\DeclareMathOperator{\rk}{rk}
\DeclareMathOperator{\id}{id}
\DeclareMathOperator{\Fix}{Fix}
\DeclareMathOperator{\Aut}{Aut}
\begin{document}

\title[Colored Stallings graphs and counterexamples to the equalizer conjecture]{Colored Stallings graphs and counterexamples to the Stallings equalizer conjecture}
\author[J.~Lei]{Jialin Lei}
\address{Department of Mathematics, Binghamton University, Binghamton, NY 13902, USA}

\email{jlei15@binghamton.edu}
\author[T.~Zhang]{Teng Zhang}
\address{School of Mathematics and Statistics, Xi'an Jiaotong University, Xi'an 710049, P. R. China}
\email{teng.zhang@stu.xjtu.edu.cn}

\subjclass[2020]{20E05, 20E07, 20F65}
\keywords{Stallings equalizer conjecture; counterexamples; colored Stallings graphs; free groups.}

\dedicatory{Dedicated to the memory of Professor J.~R. Stallings}

\thanks{All the proofs in this paper are based on (colored) Stallings graphs, reflecting the profound mathematical insight of Professor J.~R. Stallings, which has deeply inspired the authors. The second author is supported by the China Scholarship Council, the Young Elite Scientists Sponsorship Program for PhD Students (China Association for Science and Technology), and the Fundamental Research Funds for the Central Universities at Xi'an Jiaotong University (Grant No.~xzy022024045).}
\begin{abstract} 
 The famous Stallings equalizer conjecture has remained open for more than 40
years, which asserts that for any free group \(F_n\) of rank \(n\ge 2\), any free group \(F\), and any two monomorphisms $g,h:F_n\to F,$
the equalizer
$\Eq(g,h)=\{w\in F_n\mid g(w)=h(w)\}$
satisfies
$\rk \Eq(g,h)\le n.$ The only known case is $n=2$, due to A. D. Logan in 2022. By introducing the notion of colored Stallings graphs, we show that for every integer \(n\ge 2\) there exist monomorphisms
$
g,h:F_n\longrightarrow F_2
$
such that
$
\rk\Eq(g,h)\ge 2n-2.
$ This disproves the Stallings equalizer conjecture for $n\ge 3$.
\end{abstract}

\maketitle
\tableofcontents

\section{Introduction}

Let \(F(\Sigma)\) and \(F(\Delta)\) be the free groups  generated by
the sets \(\Sigma\) and \(\Delta\), respectively. Let
$ g,h:F(\Sigma)\longrightarrow F(\Delta)$
be homomorphisms. Their \emph{equalizer}  is defined by
\[
        \Eq(g,h):=\{w\in F(\Sigma):g(w)=h(w)\}.
\]
 Stallings \cite[Problems
P1 \& 5]{Sta87} proposed his famous rank-$n$ equalizer conjecture, which states the following.
 \begin{conj}[Stallings]\label{conj:Stallings}  Let $n\ge2$. Let \(F_n\) be a free group of rank \(n\), and let \(F\) be a free group. Let
$ g,h:F_n\to F
$
be monomorphisms.  Then
\[
        \rk\Eq(g,h)\le n.
\]
 \end{conj}
Goldstein and Turner~\cite{GT85} proved that under the hypotheses of Conjecture~\ref{conj:Stallings}, $\Eq(g,h)$ is  finitely generated, that is,  \(\operatorname{Eq}(g,h)\) has finite rank.  Later, they \cite{GT86} proved that under the weaker assumption that at least one of \(g\) and \(h\) is injective,
\(\operatorname{Eq}(g,h)\) has finite rank. Conjecture \ref{conj:Stallings} is also considered under such weaker assumption, see \cite[Problem F31]{BMS02},  \cite[Problem 6]{DV96} and \cite[Conjecture 8.3]{Ven02}.

In the special case where \(\Sigma=\Delta\) and $h=\id$, the equalizer is precisely the fixed subgroup.  Indeed, for any
endomorphism \(\phi:F(\Sigma)\to F(\Sigma)\),
\[
        \Fix(\phi)
        =\{w\in F(\Sigma)\mid \phi(w)=w\}
        =\Eq(\phi,\id).
\]
Thus the fixed subgroup problem is the one-map  special  case of the equalizer problem. By using train-track maps, Bestvina and Handel  \cite{BH92} proved the well-known Scott conjecture
for automorphisms, that is, Conjecture~\ref{conj:Stallings} holds in the special  case where \(F_n=F\),   $g\in\Aut (F_n)$ and $h=\id$. Later, Imrich and
Turner \cite{IT89} extended Bestvina--Handel's results to all endomorphisms. Dicks
and Ventura \cite{DV96} proved stronger inertness results for families of injective
endomorphisms, and Bergman  \cite{Ber99} extended their results to
families of endomorphisms. For related recent developments, see \cite{Jai26,LZ23}.

For general equalizers, the situation is harder.    Ventura \cite{Ven02}
noted that the injectivity hypothesis in Conjecture~\ref{conj:Stallings} is essential: without it, equalizers
may have infinite rank.
  Ciobanu and Logan \cite{CL20} proved algorithmic
and rank results for immersions, relating the problem to the Post
Correspondence Problem.  Logan \cite{Log22} proved the case $n=2$ in Conjecture~\ref{conj:Stallings} and obtained further positive results under inertness or retract
hypotheses on the images. Very recently, Carvalho and Silva \cite{CS26} investigated the fixed points and equalizers of injective endomorphisms of the free group at infinity.

The purpose of this paper is to give  a family of explicit counterexamples to Conjecture~\ref{conj:Stallings} for $n\ge 3$.
Our main result is the following.

\begin{thm}\label{thm:main}
 Let $n\ge 2$ and $F_n$ be a free group of rank $n$. Then 
 there exist monomorphisms
$
        g,h:F_n\longrightarrow F_2
$
such that
\[
        \rk\Eq(g,h)\ge 2n-2.
\] Consequently, Conjecture~\ref{conj:Stallings} is false for every $n\ge 3$.

\end{thm}
\begin{rem} 
In fact, we construct a subgroup family \(A_n\) of rank $2n-2$ such that
\[
A_n\leq_{\mathrm{ff}}\operatorname{Eq}(g,h),
\]
where \(\leq_{\mathrm{ff}}\) denotes a free factor inclusion.
In particular,
\[
\operatorname{rk}\operatorname{Eq}(g,h)\geq 2n-2.
\]
Moreover, the same family shows that the analogue of Conjecture~\ref{conj:Stallings} is false if one replaces the assumption that both \(g\) and \(h\) are injective by the assumption that \(g\) is injective and \(h\) is not, for every \(n\geq 5\); see Proposition~\ref{prop:one-noninj}.
\end{rem}
\medskip
 \noindent\textbf{Sketch of our proofs. }We introduce the notion of colored Stallings graphs for a pair $(g,h)$, which are Stallings graphs decorated by a vertex color, recording the twisted difference $g(\ell(p))^{-1}h(\ell(p))$ along path $p$ in the graph. We will see that injectivity of the color map forces the represented subgroup to be a free factor of the equalizer, giving a lower bound for the rank of the equalizer. Our examples are constructed by the following: let us consider homomorphisms with domain
\[
        F_n=\langle t,x_1,\ldots,x_{n-1}\rangle,
\]
and with codomain the rank-two free group \(F(a,b)\).   The generator $t$ creates a chain of colors
\[
        c_0,c_1,\ldots,c_{n-1},
\]
and each generator $x_i$ gives one loop at $c_0$ and one loop at $c_i$.  Thus the graph has $n$ vertices and
\[
        (n-1)+2(n-1)=3n-3
\]
positive oriented edges.  Its rank is therefore
\[
        1-n+3n-3=2n-2.
\]

\medskip
\noindent\textbf{Organization of the paper.}
In Section~\ref{s:2}, we recall the Stallings-graph formalism and introduce colored Stallings
 graphs. We prove the key embedding
criterion which detects free factors of equalizers. In Section~\ref{s:3}, we
construct the homomorphisms used in the main example and prove, by an explicit
Stallings-graph argument, that they are injective. In Section~\ref{s:4}, we
build the colored Stallings graphs \(\Gamma_n\), compute the rank of the
subgroups they represent, and prove Theorem~\ref{thm:main}. Section~\ref{s:5}
records stabilized variants in which one or both homomorphisms are allowed to
be non-injective. Finally, Section~\ref{s:6} proposes a possible sharp upper
bound suggested by the examples constructed here.

\section{Colored Stallings graphs and equalizers}\label{s:2}

Folded labeled graphs were introduced by Stallings~\cite{Sta83} in the
study of subgroups of free groups. We first recall the standard combinatorial
language of \(X\)-labeled graphs and Stallings graphs; see
Kapovich--Myasnikov~\cite[Section~2]{KM02} for background. Readers already
familiar with this material may skip to Subsection~\ref{subsec:colored-graphs}.

\subsection{\(X\)-labeled and folded graphs}
\begin{defin}[\(X\)-labeled graph]
Let \(X\) be a finite set, and let \(F(X)\) be the free group with basis
\(X\). An \emph{\(X\)-labeled graph} \(\Gamma\) consists of a vertex set
\(V\Gamma\), an oriented edge set \(E\Gamma\), maps
\[
\iota,\tau:E\Gamma\to V\Gamma
\]
recording the initial and terminal vertices, a fixed-point-free involution
\[
e\mapsto \bar e
\]
on \(E\Gamma\), and a labelling map
\[
\ell:E\Gamma\to X^{\pm1},
\]
such that
\[
\iota(\bar e)=\tau(e),\qquad
\tau(\bar e)=\iota(e),\qquad
\ell(\bar e)=\ell(e)^{-1}.
\] Additional notes are as follows:
\begin{enumerate}
    \item a \emph{topological edge} is the unordered pair \(\{e,\bar e\}\). We write
\(E^+\Gamma\) for a choice of one oriented edge from each topological edge; 
\item an \emph{oriented path} \(p\) in \(\Gamma\) is a finite sequence of oriented edges
\[
p=e_1e_2\cdots e_m
\]
such that
\[
\tau(e_j)=\iota(e_{j+1})
\]
for all \(j\). Its label is
\[
\ell(p)=\ell(e_1)\ell(e_2)\cdots \ell(e_m)\in F(X).
\]
The empty path has label \(1\). The inverse path is
\[
\bar p=\bar e_m\cdots \bar e_1,
\]
and satisfies
\[
\ell(\bar p)=\ell(p)^{-1};
\] 
    \item  a path is \emph{reduced} if it has no immediate backtracking, that is, if
\[
e_{j+1}\neq \bar e_j
\]
for all \(j\);
\item A \emph{morphism} of based \(X\)-labeled graphs is a basepoint-preserving graph
morphism which preserves initial vertices, terminal vertices, inverse edges,
and labels;
\item if \((\Gamma,v_0)\) is a connected based \(X\)-labeled graph, it represents
the subgroup
\[
L_X(\Gamma,v_0)
=
\{\ell(p)\in F(X)\mid p \text{ is a closed path in }\Gamma
\text{ based at }v_0\}.
\]
\end{enumerate}
\end{defin}

\begin{defin}[Folded graph]
    The graph \(\Gamma\) is called \emph{folded} if, for every vertex
\(v\in V\Gamma\) and every letter \(x\in X^{\pm1}\), there is at most one
oriented edge \(e\in E\Gamma\) such that
\[
\iota(e)=v,\qquad \ell(e)=x.
\]
\end{defin}

\subsection{The Stallings graph of a subgroup}

Let \(R_X\) be the rose with one vertex and one positively oriented edge
labeled by \(x\) for each \(x\in X\). We identify
\[
\pi_1(R_X)\cong F(X).
\]

Let \(H\leq F(X)\). The Schreier covering graph of \(H\), denoted $\Sigma_X(H)$, has vertices the right cosets
\[
Hg,\qquad g\in F(X),
\]
base vertex \(H\), and, for every \(x\in X\), an oriented edge
\[
Hg \xrightarrow{x} Hgx.
\]
The reverse edge has label \(x^{-1}\).
Thus every reduced word \(w\in F(X)\) determines a unique reduced path in
\(\Sigma_X(H)\) starting at the base vertex \(H\), and this path ends at the
vertex \(Hw\). In particular, this path is closed at the base vertex if and
only if $w\in H$.

The \emph{based core subgraph} of \(\Sigma_X(H)\) is the union of all reduced
closed paths in \(\Sigma_X(H)\) based at the vertex \(H\). Equivalently, it
is the subgraph consisting of all vertices and edges which occur along the
path labeled by the reduced representative of some element of \(H\). We
denote this based subgraph by $\Gamma_X(H)$.
If \(H=1\), then \(\Gamma_X(H)\) is the one-vertex graph consisting only of
the base vertex.
\begin{defin}[The Stallings graph of a subgroup]
The above based graph \(\Gamma_X(H)\) is called the \emph{Stallings graph of
\(H\) over \(X\)}. When the underlying  set and subgroup are clear, we simply call it \emph{Stallings graph}. By the former construction, \(\Gamma_X(H)\) is connected and folded,
and it represents exactly \(H\):
\[
L_X(\Gamma_X(H),\ast_H)=H,
\]
where we denote the base vertex \(H\) by \(*_H\).
Moreover, \(\Gamma_X(H)\) is finite if and only if \(H\) is finitely
generated.

\end{defin}
Now, we illustrate the correspondence between subgroups and graphs.
\begin{rem}\label{rem:subgroup-graph}
As shown above, every subgroup of \(F(X)\) has an associated based
\(X\)-labeled graph. Conversely, let \((\Gamma,v_0)\) be a connected folded based
\(X\)-labeled graph. Suppose that every topological edge of \(\Gamma\) lies
in some reduced closed path based at \(v_0\). Then
\[
(\Gamma,v_0)\cong \Gamma_X(L_X(\Gamma,v_0))
\]
as based \(X\)-labeled graphs. Thus subgroups \(H\leq F(X)\) are represented, up to based label-preserving
isomorphism, by connected folded based \(X\)-labeled graphs in which every
edge lies in a reduced closed path based at the base vertex. 
\end{rem}
The subgroup--graph correspondence is compatible with inclusions of subgroups, and the rank of a finitely generated subgroup can be read off from its Stallings graph, as recalled below.
\begin{rem}\label{rem:rank}
If
$
A\leq B\leq F(X),
$
then the map on right cosets
\[
Ag\longmapsto Bg
\]
induces a basepoint-preserving label-preserving morphism of Schreier graphs
\[
\Sigma_X(A)\longrightarrow \Sigma_X(B).
\]
Since based reduced closed paths map to based reduced closed paths, this
morphism restricts to a natural morphism of Stallings graphs
\[
\rho_{A,B}:\Gamma_X(A)\longrightarrow \Gamma_X(B).
\]
If \(\Gamma\) is finite and connected, then
\[
\operatorname{rk}\pi_1(\Gamma)
=
1-|V\Gamma|+|E^+\Gamma|.
\]
In particular, if \(H\leq F(X)\) is finitely generated, then
\[
\operatorname{rk}H
=
1-|V\Gamma_X(H)|+|E^+\Gamma_X(H)|.
\]
\end{rem}
We will also use the following elementary fact. 
\begin{rem}\label{rem:fact}
    If \(Y\) is a connected
based subgraph of a connected based graph \(Z\), then the inclusion
\[
Y\hookrightarrow Z
\]
induces a free factor inclusion
\[
\pi_1(Y)\leq_{\mathrm{ff}}\pi_1(Z).
\]
Indeed, one may choose a maximal tree in \(Y\) and extend it to a maximal
tree in \(Z\). The graph basis associated to the larger tree contains the
graph basis associated to \(Y\).
\end{rem}
\subsection{Colored \(X\)-graphs and colored Stallings graphs}\label{subsec:colored-graphs}

\begin{defin}[Colored \(X\)-graph]\label{def:colored}
Let \(G\) be a free group, and let
\[
g,h:F(X)\to G
\]
be homomorphisms. A \emph{colored \(X\)-graph for \((g,h)\)} is a connected
based \(X\)-labeled graph \((\Gamma,v_0)\), together with a map
\[
C:V\Gamma\to G,
\]
called the \emph{color map}, such that
\[
C(v_0)=1
\]
and, for every oriented edge \(e\in E\Gamma\), one has
\begin{equation}\label{eq:color-cond}
    g(\ell(e))
=
C(\iota(e))\,h(\ell(e))\,C(\tau(e))^{-1}.
\end{equation}

Here \(g\) and \(h\) are understood on letters in \(X^{\pm1}\) by
\[
g(x^{-1})=g(x)^{-1},\qquad h(x^{-1})=h(x)^{-1}.
\]
\end{defin}
It is enough to check \eqref{eq:color-cond} on one chosen orientation of each topological
edge, since the condition for the reverse orientation follows by taking
inverses. We now introduce the colored version of Stallings graphs. This device will be used to produce explicit free factors of equalizers.
\begin{defin}[Colored Stallings graph] If the underlying based graph \((\Gamma,v_0)\) in Definition~\ref{def:colored} is a Stallings
graph, then we call $(\Gamma,v_0,C)$ a \emph{colored Stallings graph}.
\end{defin}

\begin{lem}\label{lem:telescoping}
Let \((\Gamma,v_0,C)\) be a colored \(X\)-graph for \((g,h)\). Let \(p\) be
an oriented path in \(\Gamma\) from \(v_0\) to a vertex \(v\). Then
\[
C(v)=g(\ell(p))^{-1}h(\ell(p)).
\]
Consequently, every closed path in \(\Gamma\) based at \(v_0\) has label in $
\operatorname{Eq}(g,h).$
\end{lem}

\begin{proof}
We prove the first statement by induction on the length of \(p\). For the
empty path, the endpoint is \(v_0\), and the statement is precisely
\[
C(v_0)=1.
\]

Suppose the statement holds for a path \(p\) ending at a vertex \(v\). Let
\(e\) be an oriented edge with
$
\iota(e)=v, \tau(e)=w,
$
and put
$
x=\ell(e)\in X^{\pm1}.
$
By the coloring condition \eqref{eq:color-cond},
\[
g(x)=C(v)h(x)C(w)^{-1}.
\]
Equivalently,
\[
C(w)=g(x)^{-1}C(v)h(x).
\]
Using the induction hypothesis gives
\[
C(w)
=
g(x)^{-1}g(\ell(p))^{-1}h(\ell(p))h(x)
=
g(\ell(p)x)^{-1}h(\ell(p)x).
\]
Since
$
\ell(pe)=\ell(p)x,
$
this proves the induction step.

Now let \(p\) be a closed path based at \(v_0\). Then
\[
1=C(v_0)=g(\ell(p))^{-1}h(\ell(p)).
\]
Hence
\[
g(\ell(p))=h(\ell(p)),
\]
so
\[
\ell(p)\in \operatorname{Eq}(g,h).
\]
\end{proof}

\begin{lem}\label{lem:free-factor}
Let \((\Gamma,v_0,C)\) be a finite colored Stallings graph for \((g,h)\).
Assume that the color map
$
C:V\Gamma\to G
$
is injective. Put
$
A=L_X(\Gamma,v_0)\leq F(X).
$
Then
\[
A\leq \operatorname{Eq}(g,h),
\]
and the natural morphism
\[
\rho:\Gamma_X(A)\longrightarrow
\Gamma_X(\operatorname{Eq}(g,h))
\]
is injective on vertices and on edges. Consequently,
\[
A\leq_{\mathrm{ff}}\operatorname{Eq}(g,h).
\]
In particular,
\[
\operatorname{rk}\operatorname{Eq}(g,h)\geq \operatorname{rk}A.
\]
\end{lem}

\begin{proof}
By Lemma~\ref{lem:telescoping}, every closed path in \(\Gamma\) based at \(v_0\) has label in
\(\operatorname{Eq}(g,h)\). Therefore
\[
A=L_X(\Gamma,v_0)\leq \operatorname{Eq}(g,h).
\]

Since \((\Gamma,v_0)\) is a Stallings graph, Remark~\ref{rem:subgroup-graph} gives a based
label-preserving isomorphism
\[
(\Gamma,v_0)\cong \Gamma_X(A).
\]
Using this identification, the inclusion
$
A\leq \operatorname{Eq}(g,h)
$
gives the natural morphism
\[
\rho:\Gamma\longrightarrow \Gamma_X(\operatorname{Eq}(g,h)).
\]

We first prove that \(\rho\) is injective on vertices. Let
$
u,v\in V\Gamma,
$
and suppose that
$
\rho(u)=\rho(v).
$
Choose oriented paths \(p_u,p_v\) in \(\Gamma\) from \(v_0\) to \(u\) and
\(v\), respectively. Since the images of \(u\) and \(v\) coincide, the path
$
\rho(p_u)\,\overline{\rho(p_v)}
$
is a closed path based at the base vertex of
\(\Gamma_X(\operatorname{Eq}(g,h))\). Hence its label lies in
\(\operatorname{Eq}(g,h)\). Equivalently,
\[
\ell(p_u)\ell(p_v)^{-1}\in \operatorname{Eq}(g,h).
\]
Therefore
\[
g\bigl(\ell(p_u)\ell(p_v)^{-1}\bigr)
=
h\bigl(\ell(p_u)\ell(p_v)^{-1}\bigr).
\]
This gives
\[
g(\ell(p_u))g(\ell(p_v))^{-1}
=
h(\ell(p_u))h(\ell(p_v))^{-1}.
\]
Multiplying this equality on the left by \(g(\ell(p_u))^{-1}\) and on the
right by \(h(\ell(p_v))\), we obtain
\[
g(\ell(p_v))^{-1}h(\ell(p_v))
=
g(\ell(p_u))^{-1}h(\ell(p_u)).
\]
Equivalently,
\[
g(\ell(p_u))^{-1}h(\ell(p_u))
=
g(\ell(p_v))^{-1}h(\ell(p_v)).
\]
By Lemma~\ref{lem:telescoping},
\[
C(u)=C(v).
\]
Since \(C\) is injective, it follows that
\[
u=v.
\]
Thus \(\rho\) is injective on vertices.

We next prove injectivity on oriented edges. Suppose that
$
e,e'\in E\Gamma
$
and
$
\rho(e)=\rho(e').
$
Then
\[
\rho(\iota(e))=\rho(\iota(e')).
\]
By vertex injectivity,
\[
\iota(e)=\iota(e').
\]
Moreover, since \(\rho\) preserves labels,
\[
\ell(e)=\ell(e').
\]
Because \(\Gamma\) is folded, there is at most one outgoing edge with a given
label at a given vertex. Hence
\[
e=e'.
\]
Thus \(\rho\) is injective on edges.

It follows that \(\Gamma\) is identified with a connected based labeled
subgraph
\[
Y=\rho(\Gamma)\subseteq \Gamma_X(\operatorname{Eq}(g,h)).
\]
Under this identification, the subgroup represented by \(Y\) is exactly
\(A\). By the free factor statement in Remark~\ref{rem:fact}, the inclusion
\[
Y\hookrightarrow \Gamma_X(\operatorname{Eq}(g,h))
\]
induces a free factor inclusion on the represented subgroups. Hence
\[
A\leq_{\mathrm{ff}}\operatorname{Eq}(g,h).
\]
In particular,
\[
\operatorname{rk}\operatorname{Eq}(g,h)\geq \operatorname{rk}A.
\]
\end{proof}

\section{Injectivity of the construction}\label{s:3}

Fix $n\ge 2$ and put
\[
        X_n=\{t,x_1,\ldots,x_{n-1}\},
        \qquad
        F_n=\langle t,x_1,\ldots,x_{n-1}\rangle,
        \qquad
        G=F(a,b).
\]
For $0\leq i\leq n-1$,
define
\[
        c_i=a^{-i}b^i.
\]
Thus $c_0=1$.  

We shall define
\[
        g(t)=a,\qquad h(t)=b,\qquad g(x_i)=h(x_i)=c_i^2.
\]
The aim of this section is to prove, by an explicit Stallings-graph argument, that both maps are injective.

For $1\le i\le n-1$, set
\[
        d_i=b^ia^{-i}b^i.
\]

\begin{cons}\label{cons:Delta}
Let $\Delta_n$ be the following labeled graph over the alphabet $X=\{a,b\}$.
It has a base vertex $o$.  There is one $a$-loop at $o$.  There is a positive $b$-path of length $n-1$
\[
        o=P_0\xrightarrow{b}P_1\xrightarrow{b}\cdots\xrightarrow{b}P_{n-1},
\]
and a negative $b$-path of length $n-1$
\[
        o=Q_0\xrightarrow{b^{-1}}Q_1\xrightarrow{b^{-1}}\cdots\xrightarrow{b^{-1}}Q_{n-1},
\]
where the vertices \(P_i\) and \(Q_j\) are all distinct except for
\(P_0=Q_0=o\).
For each \(i=1,\ldots,n-1\), attach a new path
\[
        P_i=R_{i,0}\xrightarrow{a^{-1}}R_{i,1}
        \xrightarrow{a^{-1}}\cdots
        \xrightarrow{a^{-1}}R_{i,i}=Q_i.
\] 
The interior vertices \(R_{i,1},\ldots,R_{i,i-1}\) are new, and the
interiors of these paths are pairwise disjoint and disjoint from the two
\(b\)-paths. 
\end{cons}

\begin{lem}\label{lem:delta-rank}
The based graph \((\Delta_n,o)\) is a Stallings graph over \(X=\{a,b\}\).
Moreover,
\[
L_X(\Delta_n,o)=\langle a,d_1,\ldots,d_{n-1}\rangle\leq F(a,b),
\]
where
\[
d_i=b^ia^{-i}b^i.
\]
In particular,
\[
\operatorname{rk}L_X(\Delta_n,o)=n,
\]
and
\[
a,d_1,\ldots,d_{n-1}
\]
freely generate this subgroup.
\end{lem}

\begin{proof}
First we check that \(\Delta_n\) is folded. Along the two \(b\)-paths there
is at most one outgoing edge with each label \(b\) or \(b^{-1}\) at every
vertex. Along each attached \(a^{-i}\)-path there is at most one outgoing
edge with each label \(a\) or \(a^{-1}\). The interiors of the attached
paths are pairwise disjoint. At the base vertex \(o\), the outgoing labels
are
\[
a,\ a^{-1},\ b,\ b^{-1},
\]
which are distinct. Thus \(\Delta_n\) is folded.

Next we check the based core condition. The \(a\)-loop at \(o\) is a reduced
closed path based at \(o\). For each \(i=1,\ldots,n-1\), the path
\[
o=P_0 \xrightarrow{b^i} P_i
\xrightarrow{a^{-i}} Q_i
\xrightarrow{b^i} Q_0=o
\]
is a reduced closed path based at \(o\), and its label is
\[
b^ia^{-i}b^i=d_i.
\]
These paths contain all edges of the two \(b\)-paths and all edges of the
attached \(a^{-i}\)-paths. Hence every topological edge of \(\Delta_n\) lies
in a reduced closed path based at \(o\). Therefore \((\Delta_n,o)\) is a
Stallings graph.

It remains to identify the subgroup represented by \(\Delta_n\). Choose a
maximal tree \(T\) as follows. Put into \(T\) all edges of the two \(b\)-paths.
For each attached path
\[
P_i=R_{i,0}\to R_{i,1}\to\cdots\to R_{i,i}=Q_i,
\]
put into \(T\) all its edges except the first edge
\[
R_{i,0}=P_i\to R_{i,1}.
\]
Finally, do not put the \(a\)-loop at \(o\) into \(T\). Then the topological
edges outside \(T\) are exactly the \(a\)-loop at \(o\) and the first edge
of each attached \(a^{-i}\)-path. Hence there are \(n\) non-tree edges.

The graph basis of \(\pi_1(\Delta_n,o)\) associated to \(T\) has one
generator labelled \(a\), coming from the \(a\)-loop at \(o\). For the
omitted first edge of the \(i\)-th attached path, the corresponding based
loop goes from \(o\) to \(P_i\), crosses this omitted edge, and then returns
to \(o\) through the tree. Its label is
\[
b^i a^{-1} a^{-(i-1)} b^i
=
b^i a^{-i} b^i
=
d_i.
\]
Therefore the labels of the graph basis are precisely
\[
a,d_1,\ldots,d_{n-1}.
\]

Since \(\Delta_n\) is folded, the label map
\[
\Delta_n\longrightarrow R_{\{a,b\}}
\]
is an immersion of graphs. Hence it is injective on fundamental groups, so
\[
\pi_1(\Delta_n,o)\longrightarrow F(a,b)
\]
is injective. Hence the labels of the above graph basis freely generate the
represented subgroup. Thus
\[
L_X(\Delta_n,o)=\langle a,d_1,\ldots,d_{n-1}\rangle,
\]
and this subgroup has rank \(n\).
\end{proof}

\begin{prop}\label{prop:injective-square}
Define homomorphisms
$
        g,h:F_n\to F(a,b)
$
by
\[
        g(t)=a,\qquad h(t)=b,\qquad g(x_i)=h(x_i)=c_i^2=(a^{-i}b^i)^2
\]
for $1\le i\le n-1$.  Then both $g$ and $h$ are injective.
\end{prop}

\begin{proof}
We first prove that $g$ is injective.  For each $i$, put $y_i=t^ix_i$.
The assignment
\[
        t\mapsto t,\qquad x_i\mapsto t^i x_i\quad (1\le i\le n-1)
\]
extends to an automorphism of \(F_n\). Hence
\[
        \{t,y_1,\ldots,y_{n-1}\}
\]
is again a free basis of \(F_n\).

Now
\[
        g(y_i)=g(t)^ig(x_i)=a^i(a^{-i}b^i)^2=b^ia^{-i}b^i=d_i.
\]
Thus the image under $g$ of the free basis
\[
        \{t,y_1,\ldots,y_{n-1}\}
\]
is
\[
        \{a,d_1,\ldots,d_{n-1}\}.
\]
By Lemma~\ref{lem:delta-rank}, this set freely generates a free subgroup of rank $n$.  Hence $g$ is injective.

We now prove that \(h\) is injective. Since \(g\) is injective and
\(\{t,x_1,\ldots,x_{n-1}\}\) is a free basis of \(F_n\), the set
\[
\{a,c_1^2,\ldots,c_{n-1}^2\}
\]
is a free basis of the subgroup \(g(F_n)\). Hence
\[
\{a,c_1^{-2},\ldots,c_{n-1}^{-2}\}
\]
is also a free basis, obtained by inverting some basis elements.

Let
\[
\sigma:F(a,b)\to F(a,b),\qquad \sigma(a)=b,\quad \sigma(b)=a.
\]
Then
\[
\sigma(c_i)=\sigma(a^{-i}b^i)=b^{-i}a^i=c_i^{-1}.
\]
Therefore
\[
\sigma\bigl(\{b,c_1^2,\ldots,c_{n-1}^2\}\bigr)
=
\{a,c_1^{-2},\ldots,c_{n-1}^{-2}\}.
\]
Since the set on the right freely generates a free subgroup, so does
\[
\{b,c_1^2,\ldots,c_{n-1}^2\}.
\]
These are precisely the images under \(h\) of the basis
\[
\{t,x_1,\ldots,x_{n-1}\}.
\]
Hence \(h\) is injective.
\end{proof}

\section{The counterexample family}\label{s:4}

We now keep the maps $g$ and $h$ from Proposition~\ref{prop:injective-square} fixed:
\[
        g(t)=a,\qquad h(t)=b,\qquad g(x_i)=h(x_i)=c_i^2.
\]

\begin{cons}\label{cons:Gamma}
Let $\Gamma_n$ be the labeled graph with vertices
\[
        v_0,v_1,\ldots,v_{n-1}.
\]
The base vertex is $v_0$.  For each $1\le i\le n-1$, the graph has a directed $t$-edge
\[
        v_{i-1}\xrightarrow{t}v_i.
\]
For each $i=1,\ldots,n-1$, the graph also has two $x_i$-loops:
\[
        v_0\xrightarrow{x_i}v_0,\qquad v_i\xrightarrow{x_i}v_i.
\]
Assign the color
\[
        C(v_i)=c_i=a^{-i}b^i.
\] See Figure~\ref{fig:Gamma-n}.
\end{cons}

\begin{figure}[H]
\centering
\begin{tikzpicture}[scale=1.0, >=Stealth, vertex/.style={circle,draw,inner sep=2pt,minimum size=23pt}]
    \node[vertex] (v0) at (0,0) {$v_0$};
    \node[vertex] (v1) at (2.3,0) {$v_1$};
    \node[vertex] (v2) at (4.6,0) {$v_2$};
    \node (dots) at (6.9,0) {$\cdots$};
    \node[vertex] (vn) at (9.2,0) {$v_{n-1}$};

    \draw[->] (v0) -- node[above] {$t$} (v1);
    \draw[->] (v1) -- node[above] {$t$} (v2);
    \draw[->] (v2) -- node[above] {$t$} (dots);
    \draw[->] (dots) -- node[above] {$t$} (vn);

    \node at (-1.2,-0.50) {$\vdots$};
    \draw[->] (v0) to[out=118,in=62,looseness=9] node[above] {$x_1$} (v0);
    \draw[->] (v0) to[out=178,in=122,looseness=10] node[left] {$x_2$} (v0);
    \draw[->] (v0) to[out=238,in=302,looseness=9] node[below left] {$x_{n-1}$} (v0);

    \draw[->] (v1) to[out=120,in=60,looseness=8] node[above] {$x_1$} (v1);
    \draw[->] (v2) to[out=120,in=60,looseness=8] node[above] {$x_2$} (v2);
    \draw[->] (vn) to[out=120,in=60,looseness=8] node[above] {$x_{n-1}$} (vn);
\end{tikzpicture}
\caption{The graph $\Gamma_n$.}
\label{fig:Gamma-n}
\end{figure}
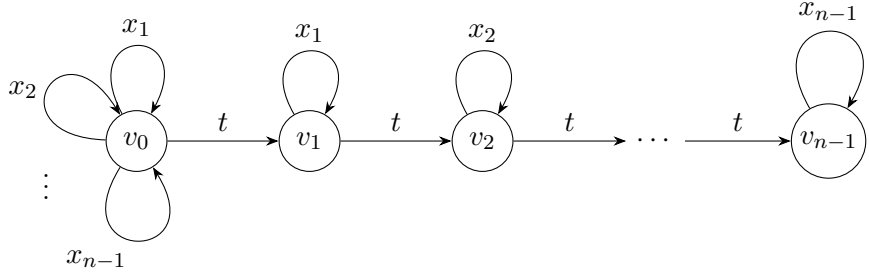

\begin{lem}\label{lem:edge-check}
The triple \((\Gamma_n,v_0,C)\) is a colored Stallings graph for the pair
\((g,h)\).
\end{lem}

\begin{proof}
By the observation following Definition~\ref{def:colored}, it suffices to
verify the edge condition
\[
        g(y)=C(v)h(y)C(w)^{-1}
\]
for the chosen orientations of the topological edges listed in
Construction~\ref{cons:Gamma}.

First consider the edge $ v_{i-1}\xrightarrow{t}v_i$.
Since
\[
        C(v_{i-1})=a^{-(i-1)}b^{i-1},
        \qquad
        C(v_i)=a^{-i}b^i,
\]
we have
\[
\begin{aligned}
        C(v_{i-1})h(t)C(v_i)^{-1}
        &=a^{-(i-1)}b^{i-1}\,b\,(a^{-i}b^i)^{-1} \\
        &=a^{-(i-1)}b^ib^{-i}a^i \\
        &=a.
\end{aligned}
\]
This is $g(t)$.

Next, we consider an $x_i$-loop at $v_0$.  Since $C(v_0)=1$, we have
\[
        C(v_0)h(x_i)C(v_0)^{-1}=c_i^2=g(x_i).
\]

Finally, consider the $x_i$-loop at $v_i$.  Since $C(v_i)=c_i$ and $h(x_i)=c_i^2$, we have
\[
        C(v_i)h(x_i)C(v_i)^{-1}=c_ic_i^2c_i^{-1}=c_i^2=g(x_i).
\]

Thus every edge satisfies the required condition.

Moreover, the underlying based graph \((\Gamma_n,v_0)\) is a Stallings graph.
It is folded: at each vertex there is at most one outgoing edge with any
prescribed label in \(X_n^{\pm1}\). Also, every edge lies in a reduced closed
path based at \(v_0\). Indeed, the \(x_i\)-loop at \(v_0\) lies in the based
loop labelled \(x_i\), whereas the \(x_i\)-loop at \(v_i\) lies in the based
loop labelled \(t^i x_i t^{-i}\). Finally, the \(t\)-edge from \(v_{j-1}\)
to \(v_j\) lies in the reduced based loop labelled \(t^j x_j t^{-j}\).
\end{proof}

\begin{lem}\label{lem:rank}
Let
$
A_n=L_{X_n}(\Gamma_n,v_0)\leq F_n
$
be the subgroup represented by \(\Gamma_n\).
Then
\[
        A_n\le_{\mathrm{ff}} \Eq(g,h)\quad \text{and}\quad
        \rk A_n=2n-2.
\]
\end{lem}
\begin{proof}
The colors are
\[
        C(v_i)=c_i=a^{-i}b^i,\qquad 0\le i\le n-1.
\]
These elements are pairwise distinct as reduced words in \(F(a,b)\). Hence
the color map \(C:V\Gamma_n\to F(a,b)\) is injective.

By Lemma~\ref{lem:edge-check}, the triple \((\Gamma_n,v_0,C)\) is a finite
colored Stallings graph for \((g,h)\). Therefore Lemma~\ref{lem:free-factor}
gives
\[
        A_n\leq_{\mathrm{ff}}\Eq(g,h).
\]

Since the underlying based graph \((\Gamma_n,v_0)\) is a finite Stallings
graph, the label map identifies \(\pi_1(\Gamma_n,v_0)\) with \(A_n\). Hence
\[
\operatorname{rk}A_n
=
\operatorname{rk}\pi_1(\Gamma_n,v_0)
=
1-|V\Gamma_n|+|E^+\Gamma_n|.
\]
The graph \(\Gamma_n\) has \(n\) vertices. It has \(n-1\) topological
\(t\)-edges, and for each \(i=1,\ldots,n-1\) it has two \(x_i\)-loops. Thus
\[
|E^+\Gamma_n|=(n-1)+2(n-1)=3n-3.
\]
Therefore
\[
        \rk A_n
        =1-n+3n-3
        =2n-2.
\]
\end{proof}
\begin{proof}[Proof of Theorem~\ref{thm:main}]
By Proposition~\ref{prop:injective-square}, both \(g\) and \(h\) are
monomorphisms. By Lemma~\ref{lem:rank}, the subgroup
$
        A_n=L_{X_n}(\Gamma_n,v_0)
$
satisfies
\[
        A_n\leq_{\mathrm{ff}}\Eq(g,h)
        \qquad\text{and}\qquad
        \rk A_n=2n-2.
\]
Hence
\[
        \rk\Eq(g,h)\geq 2n-2.
\]
If \(n\ge3\), then \(2n-2>n\), so Conjecture~\ref{conj:Stallings} is false
for every \(n\ge3\).
\end{proof}


\section{Stabilized non-injective applications}\label{s:5}

The   same construction of Sections~\ref{s:3} and~\ref{s:4}  also yields consequences when one allows non-injective maps.  The reason is that one may add dummy generators without changing the large subgroup already lying in the equalizer.

\begin{prop}\label{prop:one-noninj}
For every integer $r\ge 5$, there exist homomorphisms
\[
        \alpha,\beta:F_r\to F
\]
to a free group $F$ such that $\alpha$ is injective, $\beta$ is not injective, and
\[
        \rk\Eq(\alpha,\beta)>r.
\]
\end{prop}

\begin{proof} Put \(m=r-1\). Since \(r\ge 5\), we have \(m\ge 4\). 
Let $X_m=\{t,x_1,\ldots,x_{m-1}\}$ and $X_r=X_m\sqcup\{z\}.$
Apply
Theorem~\ref{thm:main} to obtain monomorphisms
\[
        g,h:F_m\to F_2
\]
and the colored Stallings graph \((\Gamma_m,v_0,C)\) whose represented
subgroup
\[
        A_m=L_{X_m}(\Gamma_m,v_0)
\]
satisfies
\[
        A_m\leq_{\mathrm{ff}}\Eq(g,h),\qquad \rk A_m=2m-2.
\]

Write
\[
        F_r=F_m*\langle z\rangle.
\]
Let \(s\) be a new free generator and let the target be
\[
        F=F_2*\langle s\rangle.
\]
Define
\[
        \alpha|_{F_m}=g,\qquad \alpha(z)=s,
\]
and
\[
        \beta|_{F_m}=h,\qquad \beta(z)=1.
\]
Then \(\alpha\) is injective, since it is the free product of the
monomorphism \(g\) with the embedding of \(\langle z\rangle\) as
\(\langle s\rangle\). The map \(\beta\) is not injective, because
\(z\in\ker\beta\).

It remains to justify the rank estimate. Since \(A_m\le F_m\le F_r=F(X_r)\), the \(X_r\)-Stallings graph of
\(A_m\) is the same graph \((\Gamma_m,v_0)\). Indeed, by the normal form theorem for free products, a reduced word in the
enlarged basis \(X_r\) which contains \(z^{\pm1}\) represents an element outside the free factor \(F_m\).

Moreover, \((\Gamma_m,v_0,C)\) is a colored \(X_r\)-Stallings graph for
\((\alpha,\beta)\). All edges of \(\Gamma_m\) are labelled by letters in
\(X_m\), and on those edges the coloring equations are exactly the
coloring equations for \((g,h)\). The color map remains injective after
viewing its values in \(F_2*\langle s\rangle\).

Therefore Lemma~\ref{lem:free-factor} gives
\[
        A_m\leq_{\mathrm{ff}}\Eq(\alpha,\beta).
\]
Hence
\[
        \rk\Eq(\alpha,\beta)\ge \rk A_m=2m-2=2(r-1)-2=2r-4.
\]
For \(r\ge 5\), one has \(2r-4>r\).
\end{proof}

\begin{prop}\label{prop:two-noninj}
For every integer $r\ge 7$, there exist non-injective homomorphisms
\[
        \alpha,\beta:F_r\to F
\]
to a free group $F$ such that
\[
        \rk\Eq(\alpha,\beta)>r.
\]
\end{prop}

\begin{proof}
Put \(m=r-2\). Since \(r\ge 7\), we have \(m\ge 5\). Let $X_m=\{t,x_1,\ldots,x_{m-1}\}$ and $X_r=X_m\sqcup\{z_1,z_2\}$.
  Apply
Theorem~\ref{thm:main} to obtain monomorphisms
\[
        g,h:F_m\to F_2
\]
and the colored Stallings graph \((\Gamma_m,v_0,C)\) whose represented
subgroup
\[
        A_m=L_{X_m}(\Gamma_m,v_0)
\]
has rank
\[
        \rk A_m=2m-2.
\]

Write
\[
        F_r=F_m*\langle z_1\rangle*\langle z_2\rangle.
\]
Let \(s\) be a new free generator and let the target be
\[
        F=F_2*\langle s\rangle.
\]
Define
\[
        \alpha|_{F_m}=g,\qquad \beta|_{F_m}=h,
\]
and set
\[
        \alpha(z_1)=1,\qquad \beta(z_1)=s,
\]
\[
        \alpha(z_2)=s,\qquad \beta(z_2)=1.
\]
Then \(\alpha\) is not injective, because \(z_1\in\ker\alpha\), and
\(\beta\) is not injective, because \(z_2\in\ker\beta\).
As in the proof of Proposition~\ref{prop:one-noninj}, the \(X_r\)-Stallings
graph of \(A_m\) is still \((\Gamma_m,v_0)\), because, by the normal form theorem for free products, no reduced word
involving \(z_1^{\pm1}\) or \(z_2^{\pm1}\) can represent an element of
the free factor \(F_m\).

The same coloring \(C\) makes \((\Gamma_m,v_0,C)\) a colored
\(X_r\)-Stallings graph for \((\alpha,\beta)\), and the color map is
still injective. Lemma~\ref{lem:free-factor} therefore gives
\[
        A_m\leq_{\mathrm{ff}}\Eq(\alpha,\beta).
\]
Consequently
\[
        \rk\Eq(\alpha,\beta)\ge \rk A_m=2m-2=2(r-2)-2=2r-6.
\]
For \(r\ge 7\), one has \(2r-6>r\).
\end{proof}

\begin{rem}\label{rem:inf-rank}
Once both maps are allowed to be non-injective, the equalizer may even fail to be finitely generated.  Here is a Stallings-graph way to see this.

Let
\[
        F(x,y)=\langle x,y\rangle,
        \qquad
        F(a)=\langle a\rangle,
\]
and define
\[
        \alpha(x)=a,\qquad \alpha(y)=1,
        \qquad
        \beta(x)=1,\qquad \beta(y)=a.
\]
Then both $\alpha$ and $\beta$ are non-injective.  If $e_x(w)$ and $e_y(w)$ denote the exponent sums of $x$ and $y$ in a word $w\in F(x,y)$, then
\[
        \alpha(w)=a^{e_x(w)},\qquad \beta(w)=a^{e_y(w)}.
\]
Hence
\[
        w\in\Eq(\alpha,\beta)
        \quad\Longleftrightarrow\quad
        e_x(w)=e_y(w).
\]
Equivalently,
\[
        \Eq(\alpha,\beta)=\ker\phi,
\]
where
\[
        \phi:F(x,y)\to\mathbb Z,\qquad
        \phi(x)=1,\qquad \phi(y)=-1.
\]

The Stallings graph of $\ker\phi$ is the covering of the two-petal rose corresponding to this epimorphism onto $\mathbb Z$.  Its vertices are indexed by the integers.  From the vertex $k$ there is an $x$-edge to $k+1$ and a $y$-edge to $k-1$.  Thus, between two consecutive vertices $k$ and $k+1$, there are two distinct unoriented edges: the $x$-edge from $k$ to $k+1$ and the $y$-edge from $k+1$ to $k$.  This infinite folded Stallings graph has one independent cycle between every consecutive pair of vertices.  Therefore its first Betti number is infinite.  Hence $\Eq(\alpha,\beta)$ is a free group of countably infinite rank and is not finitely generated.
\end{rem}

\section{A possible sharp bound}\label{s:6}

The family above suggests the following candidate for the correct sharp upper bound in the monomorphism case.

\begin{conj}\label{conj:2n-2}
Let $n\ge3$, and let $g,h:F_n\to F$
be monomorphisms into a free group.  Then
\[
        \rk\Eq(g,h)\le 2n-2.
\]
\end{conj}

Theorem~\ref{thm:main} shows that, if this conjecture is true, then it is sharp.



\section*{Acknowledgments}
We are grateful to Professors George Mark Bergman, Matthew G. Brin, Andrei Jaikin-Zapirain, Alan D. Logan, Lorenzo Ruffoni, Pedro Silva, Edward C. Turner, and Enric Ventura Capell for their valuable comments. We also thank Professor Qiang Zhang for bringing this problem to our attention, and Professors Minghua Lin, and Weihua Yang for their support and encouragement.

\end{document}